\documentclass[11pt, a4paper]{article}
\usepackage{typearea}
\typearea{12}

\usepackage{amsmath, amssymb, enumerate, amsthm, accents}
\usepackage{amsmath, amssymb, url, comment}
\usepackage[numbers]{natbib}
\usepackage{tikz}
\usepackage{pgfplotstable}
\usetikzlibrary{arrows,positioning,plotmarks,external,patterns,angles,decorations.pathmorphing,backgrounds,fit,shapes,graphs,calc,spy}
\pgfplotsset{compat=1.14}

\theoremstyle{plain}
\newtheorem{theorem}{Theorem}[section]
\newtheorem{lemma}[theorem]{Lemma}
\newtheorem{corollary}[theorem]{Corollary}
\newtheorem{proposition}[theorem]{Proposition}
\newtheorem{definition}[theorem]{Definition}
\theoremstyle{remark}
\newtheorem{remark}[theorem]{Remark}

\begin{document}
	\title{Estimation under matrix quadratic loss \\and matrix superharmonicity}
	\author{Takeru~Matsuda\thanks{RIKEN Center for Brain Science, e-mail: \texttt{takeru.matsuda@riken.jp}}
	~and~William~E.~Strawderman\thanks{Department of Statistics and Biostatistics, Rutgers University}
	}
	\date{}
	
	\maketitle

		\begin{abstract}
	We investigate estimation of a normal mean matrix under the \textit{matrix quadratic loss}.
	Improved estimation under the matrix quadratic loss implies improved estimation of any linear combination of the columns. 
	First, an unbiased estimate of risk is derived and the Efron–Morris estimator is shown to be minimax. 
	Next, a notion of matrix superharmonicity for matrix-variate functions is introduced and shown to have analogous properties with usual superharmonic functions, which may be of independent interest. 
	Then, we show that the generalized Bayes estimator with respect to a matrix superharmonic prior is minimax. 
	We also provide a class of matrix superharmonic priors that includes the previously proposed generalization of Stein’s prior. 
	Numerical results demonstrate that matrix superharmonic priors work well for low rank matrices.	
\end{abstract}

\section{Introduction}
Suppose that we have a matrix observation $X \in \mathbb{R}^{n \times p}$ whose entries are independent normal random variables $X_{ij} \sim {\rm N} (M_{ij},1)$, where $n-p-1>0$ and $M \in \mathbb{R}^{n \times p}$ is an unknown mean matrix. 	
By using the notation of matrix-variate normal distributions, it is expressed as $X \sim {\rm N}_{n,p}(M,I_n,I_p)$, where $I_n$ is the $n$-dimensional identity matrix.
In this setting, we consider estimation of $M$ under the \textit{matrix quadratic loss} \citep{abu,Bilodeau,Honda,Xie}:
\begin{align*}
L(M,\hat{M}) = (\hat{M} - M)^{\top} (\hat{M} - M),
\end{align*}
which takes a value in the set of $p \times p$ positive semidefinite matrices.
Namely, the risk function of an estimator $\hat{M}=\hat{M}(X)$ is defined as
\begin{align*}
R(M,\hat{M}) = {\rm E}_M [L(M,\hat{M}(X))],
\end{align*}
and an estimator $\hat{M}_1$ is said to dominate another estimator $\hat{M}_2$ if $R(M,\hat{M}_1) \preceq R(M,\hat{M}_2)$ for every $M$, where $\preceq$ is the L\"{o}wner order: $A \preceq B$ means that $B-A$ is positive semidefinite.
Thus, if $\hat{M}_1$ dominates $\hat{M}_2$, then
\begin{align*}
{\rm E}_M [ \| (\hat{M}_1 - M) c \|^2] \leq {\rm E}_M [\| (\hat{M}_2 - M) c \|^2]
\end{align*}
for every $M$ and $c \in \mathbb{R}^p$.
In particular, each column of $\hat{M}_1$ dominates that of $\hat{M}_2$ as an estimator of the corresponding column of $M$ under quadratic loss.

In the context of multivariate linear regression, improved estimation of the regression coefficient matrix under the matrix quadratic loss implies improved estimation of mean response for any value of the explanatory variables.
Specifically, let $x_i \in \mathbb{R}^p$ and $y_i \sim {\rm N}_q(B^{\top} x_i, \Sigma)$ for $i=1,\dots,n$, where $B \in \mathbb{R}^{p \times q}$ and $\Sigma \in \mathbb{R}^{q \times q}$ is a known covariance matrix.
By using the notation of matrix-variate normal distributions, it is expressed as $Y \sim {\rm N}_{n,q} (X B, I_n, \Sigma)$, where $X = (x_1, \dots, x_n)^{\top} \in \mathbb{R}^{n \times p}$ and $Y = (y_1, \dots, y_n)^{\top} \in \mathbb{R}^{n \times q}$.
Then, $\bar{B}=(X^{\top}X)^{-1} X^{\top} Y$ is sufficient for $B$ and its distribution is $\bar{B} \sim {\rm N}_{p,q} (B, (X^{\top} X)^{-1}, \Sigma)$.
Thus, multivariate linear regression reduces to estimation of $B$ based on $\bar{B}$.
If $\hat{B}_1$ dominates $\hat{B}_2$ under the matrix quadratic loss $L(B,\hat{B})=(\hat{B}-B)(\hat{B}-B)^{\top}$, then $\hat{B}_1$ provides better estimation of mean response $B^{\top} x$ for any value of the explanatory variables $x$: $E \{ \| (\hat{B}_1 - B)^{\top} x \|^2 \} \leq E \{ \| (\hat{B}_2 - B)^{\top} x \|^2 \}$.

Whereas we adopt the matrix quadratic loss in this study, most existing studies on estimation of a normal mean matrix used the Frobenius loss:
\begin{align*}
l(M,\hat{M}) = \| \hat{M} - M \|_{{\rm F}}^2 = \sum_{a=1}^n \sum_{i=1}^p (\hat{M}_{ai} - M_{ai})^2,
\end{align*}
which is the trace of the matrix quadratic loss: $l(M,\hat{M}) = {\rm tr} \ L(M,\hat{M})$.
When $n-p-1>0$, Efron and Morris~\cite{Efron72} proposed an estimator
\begin{align}
\hat{M}_{{\rm EM}} = X \left( I_p-(n-p-1) (X^{\top} X)^{-1} \right), \label{EM_estimator}
\end{align}
which is viewed as a matrix version of the James--Stein estimator.
They showed that $\hat{M}_{{\rm EM}}$ is minimax and dominates the maximum likelihood estimator $\hat{M}=X$ under the Frobenius loss.
Note that $\hat{M}_{{\rm EM}}$ does not change the singular vectors but shrinks the singular values of $X$ towards zero \citep{Stein74}.
Motivated from this property, Matsuda and Komaki~\cite{Matsuda} developed a singular value shrinkage prior 
\begin{align}
\pi_{{\rm SVS}} (M)=\det (M^{\top} M)^{-(n-p-1)/2}, \label{SVS}
\end{align} 
which is viewed as a matrix version of Stein's prior for a normal mean vector \citep{Stein74}.
They proved that this prior is superharmonic and thus the generalized Bayes estimator with respect to $\pi_{{\rm SVS}}$ is minimax and dominates the maximum likelihood estimator under the Frobenius loss.
Both the Efron--Morris estimator and the generalized Bayes estimator with respect to $\pi_{{\rm SVS}}$ have large risk reduction when $M$ is close to low rank, because they shrink the singular values separately.

In this study, we investigate shrinkage estimation of a normal mean matrix under the matrix quadratic loss.
First, we derive an unbiased estimate of matrix quadratic risk for general estimators and prove that the Efron--Morris estimator \eqref{EM_estimator} is minimax.
Next, we introduce a notion of \textit{matrix superharmonicity} for matrix-variate functions and show that it has analogous properties with usual superharmonic functions \cite{Helms}, which may be of independent interest.
Then, we show that the generalized Bayes estimator with respect to a matrix superharmonic prior is minimax.
For a normal mean vector ($p=1$), it reduces to the result by Stein~\cite{Stein74} on quadratic loss.
We also provide a class of matrix superharmonic priors, which includes the singular value shrinkage prior \eqref{SVS}.
Finally, we give several numerical results, which demonstrate that matrix superharmonic priors work well for low rank matrices.

\section{Estimation under matrix quadratic loss}
In this section, we provide preliminary results on estimation of a normal mean matrix under the matrix quadratic loss.

\subsection{Definition of minimaxity}
Since the matrix quadratic loss is only partially ordered, the definition of minimaxity under this loss is not straightforward.
In this paper, we adopt the following definition.

\begin{definition}
	An estimator $\hat{M}$ is said to be minimax under the matrix quadratic loss if
	\begin{align*}
	\sup_M \ c^{\top} R(M,\hat{M}) c = \inf_{\hat{M}'} \sup_M \ c^{\top} R(M,\hat{M}') c
	\end{align*}
	for every $c \in \mathbb{R}^p$.
\end{definition}

Namely, we define $\hat{M}$ to be minimax under the matrix quadratic loss if $\hat{M} c$ is a minimax estimator of $Mc$ under quadratic loss for every $c$, because $c^{\top} R(M,\hat{M}) c = {\rm E}_M [\| \hat{M}c-Mc \|^2]$.
In particular, since $Xc$ is a minimax estimator of $Mc$ for every $c$, the maximum likelihood estimator $\hat{M}=X$ is minimax under the matrix quadratic loss with constant risk $R(M,\hat{M})=n I_p$.
Thus, any estimator that dominates the maximum likelihood estimator is also minimax.

\subsection{Unbiased estimate of risk}
Here, we derive an unbiased estimate of the matrix quadratic risk.
Our derivation is based on Stein's lemma for matrix-variate normal distributions expressed by a matrix version of divergence.
Note that a locally integrable function $g: \mathbb{R}^{n \times p} \to \mathbb{R}^{n \times p}$ is said to be weakly differentiable \citep{shr_book} if there exist locally integrable functions $\partial_{ai} h$ for $a =1,\dots,n$ and $i=1,\dots,p$ such that
\begin{align*}
	\int h(X) \frac{\partial}{\partial X_{ai}} \phi(X) {\rm d} X = -\int \partial_{ai} h(X) \phi(X) {\rm d} X
\end{align*}
holds for any infinitely differentiable function $\phi: \mathbb{R}^{n \times p} \to \mathbb{R}$ with compact support.

\begin{definition}
	For a function $g: \mathbb{R}^{n \times p} \to \mathbb{R}^{n \times p}$, its {matrix divergence} $\widetilde{\mathrm{div}} \ g: \mathbb{R}^{n \times p} \to \mathbb{R}^{p \times p}$ is defined as 
	\begin{align*}
	(\widetilde{\mathrm{div}} \ g(X))_{ij} = \sum_{a=1}^n \frac{\partial}{\partial X_{ai}} g_{aj}(X).
	\end{align*}
\end{definition}

\begin{lemma}\label{lem_sure}
	Let $X \sim {\rm N}_{n,p}(M,I_n,I_p)$ and $g: \mathbb{R}^{n \times p} \to \mathbb{R}^{n \times p}$ be a weakly differentiable function.
	Then,
	\begin{align*}
	{\rm E}_M [(X-M)^{\top} g(X)] = {\rm E}_M [\widetilde{\mathrm{div}} \ g(X)].
	\end{align*}
\end{lemma}
\begin{proof}
    By applying Stein's lemma \cite{shr_book},
    \begin{align*}
	{\rm E}_M [(X-M)^{\top} g(X)]_{ij} &= {\rm E}_M \left[ \sum_a (X_{ai}-M_{ai}) g_{aj} (X) \right] \\
	&= {\rm E}_M \left[ \sum_a \frac{\partial}{\partial X_{ai}} g_{aj} (X) \right] \\
	&= {\rm E}_M \left[ (\widetilde{\mathrm{div}} g(X))_{ij} \right].
	\end{align*}
\end{proof}

\begin{theorem}\label{th_sure}
	The matrix quadratic risk of an estimator $\hat{M}=X+g(X)$ with a weakly differentiable function $g$ is given by
	\begin{align*}
	R(M,\hat{M}) = n I_p + {\rm E}_M [\widetilde{\mathrm{div}} \ g(X) + (\widetilde{\mathrm{div}} \ g(X))^{\top} + g(X)^{\top} g(X)].
	\end{align*}
\end{theorem}
\begin{proof}
	By using Lemma~\ref{lem_sure},
	\begin{align*}
	&R(M,\hat{M}) \\
	=& {\rm E}_M [(X+g(X)-M)^{\top} (X+g(X)-M)] \nonumber \\
	=& {\rm E}_M [(X-M)^{\top} (X-M) + (X-M)^{\top} g(X) + g(X)^{\top} (X-M) + g(X)^{\top} g(X)] \nonumber \\
	=& n I_p + {\rm E}_M [\widetilde{\mathrm{div}} g(X) + (\widetilde{\mathrm{div}} g(X))^{\top} + g(X)^{\top} g(X)].
	\end{align*}
\end{proof}

\subsection{Minimaxity of Efron--Morris estimator}
By using Theorem~\ref{th_sure}, we show that the Efron--Morris estimator \eqref{EM_estimator} is minimax under the matrix quadratic loss.

\begin{theorem}
	When $n-p-1>0$, the Efron--Morris estimator $\hat{M}_{\mathrm{EM}}$ in \eqref{EM_estimator} is minimax under the matrix quadratic loss.
\end{theorem}
\begin{proof}
	Let $g(X) = -(n-p-1) X (X^{\top} X)^{-1}$ so that $\hat{M}_{{\rm EM}}=X+g(X)$.

	To calculate $\widetilde{\mathrm{div}} \ g(X)$, we use the formula \cite{Magnus}
	\begin{align*}
	\frac{\partial}{\partial X_{ai}} (X^{\top} X)^{-1} = -(X^{\top} X)^{-1} (X^{\top} E_{ai} + E_{ai}^{\top} X) (X^{\top} X)^{-1},
	\end{align*}
	where $E_{ai} \in \mathbb{R}^{n \times p}$ is the matrix unit with 1 in the $(a,i)$-th entry and 0s elsewhere.
	Then,
	\begin{align*}
	&\frac{\partial}{\partial X_{ai}} (X (X^{\top} X)^{-1})_{aj} \\
	=& \frac{\partial}{\partial X_{ai}} \sum_k X_{ak} ((X^{\top} X)^{-1})_{kj} \\
	=& ((X^{\top} X)^{-1})_{ij} - \sum_{k,l,m}  X_{ak} ((X^{\top} X)^{-1})_{kl} (X^{\top} E_{ai} + E_{ai}^{\top} X)_{lm} ((X^{\top} X)^{-1})_{mj} \\
	=& ((X^{\top} X)^{-1})_{ij} - \sum_{k,l,m}  X_{ak} ((X^{\top} X)^{-1})_{kl} (\delta_{im} X_{al} + \delta_{il} X_{am}) ((X^{\top} X)^{-1})_{mj},
	\end{align*}
    where $\delta_{ij}$ is the Kronecker delta.
	Thus,
	\begin{align*}
	&\sum_a \frac{\partial}{\partial X_{ai}} (X (X^{\top} X)^{-1})_{aj} \\
	=& n ((X^{\top} X)^{-1})_{ij} - \sum_{k,l,m} ((X^{\top} X)^{-1})_{kl} (\delta_{im} (X^{\top} X)_{kl} + \delta_{il} (X^{\top} X)_{km}) ((X^{\top} X)^{-1})_{mj} \\
	=& (n-p-1) ((X^{\top} X)^{-1})_{ij}.	
	\end{align*}
	Therefore,
	\begin{align*}
	\widetilde{\mathrm{div}} \ g(X) = - (n-p-1)^2 (X^{\top} X)^{-1}.
	\end{align*}
	
	Also,
	\begin{align*}
	g(X)^{\top} g(X) = (n-p-1)^2 (X^{\top} X)^{-1}.
	\end{align*}
	
	Therefore, from Theorem~\ref{th_sure}, 
	\begin{align}\label{EM_risk}
	R(M,\hat{M}_{\mathrm{EM}}) = n I_p - (n-p-1)^2 {\rm E}_M [(X^{\top} X)^{-1}] \preceq n I_p,
	\end{align}
	which means that $\hat{M}_{{\rm EM}}$ is minimax.
\end{proof}

\begin{corollary}\label{cor_EM}
    The matrix quadratic risk of the Efron--Morris estimator at $M=O$ is 
    \begin{align*}
        R(O,\hat{M}_{\mathrm{EM}}) = (p+1) I_p.        
    \end{align*}
\end{corollary}
\begin{proof}
    When $M=O$, the matrix $(X^{\top} X)^{-1}$ follows the inverse Wishart distribution $W^{-1}(n,I_p)$.
    Thus, from the formula for the mean of the inverse Wishart distribution \cite{Gupta},
	\begin{align*}
	{\rm E}_{M=O} [(X^{\top} X)^{-1}] = (n-p-1)^{-1} I_p.
	\end{align*}
    Therefore, from \eqref{EM_risk},
	\begin{align*}
	R(O,\hat{M}_{\mathrm{EM}}) = n I_p - (n-p-1)^2 {\rm E}_{M=O} [(X^{\top} X)^{-1}] = (p+1) I_p.
	\end{align*}
\end{proof}

We will study the matrix quadratic risk of the Efron--Morris estimator numerically in Section 5.

\begin{remark}
	Matsuda and Komaki~\cite{Matsuda2} developed an empirical Bayes method for matrix completion based on the Efron--Morris estimator.
	It is an interesting future problem to investigate its performance in terms of the matrix quadratic loss.
\end{remark}

\section{Matrix superharmonicity}
In this section, we introduce a notion of \textit{matrix superharmonicity} for matrix-variate functions, which will be used to derive Bayes minimax estimators in the next section.

First, we review the definition of a superharmonic function \cite[Definition~3.3.3]{Helms}.
Let $S_{x,r}=\{ x+re \mid e \in \mathbb{R}^n, \| e \| = 1 \} \subset \mathbb{R}^n$ be the sphere with center $x$ and radius $r>0$.
For a function $f: \mathbb{R}^{n} \to \mathbb{R} \cup \{ \infty \}$, let
\begin{align*}
L(f: x, r)= \frac{1}{\Omega_n r^{n-1}} \int_{S_{x,r}} f(z) {\rm d} s(z) = \frac{1}{\Omega_n} \int_{S_{0,1}} f(x+re) {\rm d} s(e)
\end{align*}
be the average value of $f$ on $S_{x,r}$, where ${\rm d} s$ denotes the surface area element and $\Omega_n$ is the surface area of the unit sphere $S_{0,1}$ in $\mathbb{R}^n$.
Then, $f$ is said to be superharmonic if it satisfies the following:

\begin{enumerate}
	\item $f$ is lower semicontinuous;
	
	\item $f \not\equiv \infty$;
	
	\item $L(f: x, r)\leq f(x)$ for every $x \in \mathbb{R}^n$ and $r>0$.
\end{enumerate}

Now, we introduce matrix superharmonicity.
For a matrix-variate function $f: \mathbb{R}^{n \times p} \to \mathbb{R} \cup \{ \infty \}$, we can define its superharmonicity as the superharmonicity of $f \circ \mathrm{vec}^{-1}: \mathbb{R}^{np} \to \mathbb{R} \cup \{ \infty \}$, where $\mathrm{vec}: \mathbb{R}^{n \times p} \to \mathbb{R}^{np}$ is the vectorization operator \cite{Gupta}.
However, such a definition does not take into account the matrix structure.
We propose a stronger version of superharmonicity for matrix-variate functions, which we refer to as matrix superharmonicity.
For $X \in \mathbb{R}^{n \times p}$ and $\rho \in \mathbb{R}^p$, let $S_{X,\rho}=\{ X+e \rho^{\top} \mid e \in \mathbb{R}^n, \| e \| = 1 \}$ and
\begin{align*}
L(f: X, \rho)= \frac{1}{\Omega_n} \int_{S_{0,1}} f(X+e \rho^{\top}) {\rm d} s(e)
\end{align*}
be the average value of $f$ on $S_{X,\rho}$.
{Note that we take average over only rank one perturbations around $X$ here.}

\begin{definition}
	A matrix-variate function $f: \mathbb{R}^{n \times p} \to \mathbb{R} \cup \{ \infty \}$ is said to be matrix superharmonic if it satisfies the following:
	
	\begin{enumerate}
		\item $f$ is lower semicontinuous;
		
		\item $f \not\equiv \infty$;
		
		\item $L(f: X, \rho) \leq f(X)$ for every $X \in \mathbb{R}^{n \times p}$ and $\rho \in \mathbb{R}^p$.
	\end{enumerate}
\end{definition}

We will provide examples of matrix superharmonic functions in Sections 4.3 and 4.4.

\begin{proposition}\label{prop_super}
	If a function $f: \mathbb{R}^{n \times p} \to \mathbb{R} \cup \{ \infty \}$ is matrix superharmonic, then $f \circ \mathrm{vec}^{-1}$ is superharmonic.
\end{proposition}
\begin{proof}
	For every $X \in \mathbb{R}^{n \times p}$ and $r > 0$, 
\begin{align*}
	L(f \circ \mathrm{vec}^{-1}: \mathrm{vec} (X), r) = \frac{1}{\Omega_{p} r^{p-1}} \int_{S_{0,r}} L(f: X, \rho) {\rm d} s(\rho).
\end{align*}
    Since $L(f: X, \rho) \leq f(X)$ for every $\rho$ from the matrix superharmonicity of $f$,
\begin{align*}
	L(f \circ \mathrm{vec}^{-1}: \mathrm{vec} (X), r) \leq f(X).
\end{align*}
\end{proof}

The converse of Proposition~\ref{prop_super} does not hold when $p \geq 2$.
One counterexample is $f(X)=\| X \|_{\mathrm{F}}^{2-np}$ as we will show in Proposition~\ref{prop_stein}.

\begin{remark}
When $n=1$, $S_{X,\rho}=\{ X-\rho, X+\rho \}$ and thus the third condition of matrix superharmonicity reduces to midpoint convexity:
\begin{align*}
	L(f: X, \rho)= \frac{f(X+\rho)+f(X-\rho)}{2} \leq f(X),
\end{align*}
	which is equivalent to usual convexity for a lower semicontinuous function $f$ on $\mathbb{R}^p$.
	On the other hand, the definition of midpoint convexity is not unique for functions defined on discrete space \cite{Murota}.
	It may be interesting to investigate discrete analogue of matrix superharmonicity and its applications.
\end{remark}

We provide a characterization of matrix superharmonicity for $C^2$ functions.
Recall that a $C^2$ function $f: \mathbb{R}^{n} \to \mathbb{R}$ is superharmonic if and only if its Laplacian is nonpositive: 	    
\begin{align*}
	{\Delta} f(x) = \sum_{a=1}^n \frac{\partial^2}{\partial x_a^2} f(x) \leq 0
	\end{align*}
for every $x$ \cite[Lemma~3.3.4]{Helms}.
This property is extended to matrix superharmonic functions by using a matrix version of the Laplacian.

\begin{definition}
	For a $C^2$ function $f: \mathbb{R}^{n \times p} \to \mathbb{R}$, its {matrix Laplacian} $\widetilde{\Delta} f: \mathbb{R}^{n \times p} \to \mathbb{R}^{p \times p}$ is defined as 
	\begin{align*}
	(\widetilde{\Delta} f(X))_{ij} = \sum_{a=1}^n \frac{\partial^2}{\partial X_{ai} \partial X_{aj}} f(X).
	\end{align*}
\end{definition}

\begin{theorem}\label{th_c2}
	A $C^2$ function $f: \mathbb{R}^{n \times p} \to \mathbb{R}$ is matrix superharmonic if and only if its matrix Laplacian is negative semidefinite $\widetilde{\Delta} f(X) \preceq O$ for every $X$.
\end{theorem}
\begin{proof}
	Assume that $f$ is matrix superharmonic.
	Let $e \in \mathbb{R}^n$ be a unit vector.
	From Taylor's theorem,
	\begin{align*}
	&f(X+e \rho^{\top}) \\
	=& f(X) + \sum_{a,i} \frac{\partial f}{\partial X_{ai}} (X) e_a \rho_i + \frac{1}{2} \sum_{a,i,b,j} \frac{\partial^2 f}{\partial X_{ai} \partial X_{bj}} (X) e_a \rho_i e_b \rho_j + o(\| \rho \|^2).
	\end{align*}
	Thus, 
	\begin{align*}
L(f : X, \rho) &= \frac{1}{\Omega_n} \int_{S_{0,1}} f(X+e \rho^{\top}) {\rm d} s(e) \\
	&= f(X) + \frac{1}{2 n} \sum_{a,i,j} \frac{\partial^2 f}{\partial X_{ai} \partial X_{aj}} (X) \rho_i \rho_j + o(\| \rho \|^2) \\
    &= f(X) + \frac{1}{2n} \rho^{\top} \widetilde{\Delta} f(X) \rho + o(\| \rho \|^2),
    \end{align*}
    where we used
	\begin{align*}
	\frac{1}{\Omega_n} \int_{S_{0,1}} e_a {\rm d} s(e) = 0, \quad \frac{1}{\Omega_n} \int_{S_{0,1}} e_a e_b {\rm d} s(e) = \frac{1}{n} \delta_{ab}.
    \end{align*}
    On the other hand, $L(f: X, \rho) \leq f(X)$ from the matrix superharmonicity of $f$.
    Therefore, we must have 
	\begin{align*}
	\rho^{\top} \widetilde{\Delta} f(X) \rho \leq 0
	\end{align*}
    for every $\rho$.
    Hence, $\widetilde{\Delta} f(X) \preceq O$ for every $X$.

	Conversely, assume that $\widetilde{\Delta} f(X) \preceq O$ for every $X$.
    For arbitrary $X$ and $\rho$, let $\bar{f}(e)=f(X+e \rho^{\top})$.
	Then,
	\begin{align*}
	{\Delta} \bar{f}(e) &= \sum_a \frac{\partial^2 \bar{f}}{\partial e_a^2} = \sum_{a,i} \frac{\partial}{\partial e_a} \left( \frac{\partial f}{\partial X_{ai}} \rho_i \right) = \sum_{a,i,j} \rho_i \frac{\partial^2 f}{\partial X_{ai} \partial X_{aj}} \rho_j = \rho^{\top} \widetilde{\Delta} f (X+e \rho^{\top}) \rho \leq 0.
	\end{align*}
	Let $C_{\eta} = \{ e \in \mathbb{R}^n \mid \eta < \| e \| < 1 \}$ and $\bar{g}(e)=\| e \|^{2-n}-1$.
	Then, $\bar{g} \geq 0$ and $\Delta \bar{g}=0$ on $C_{\eta}$.
	Therefore, from Green's theorem,
	\begin{align*}
	\int_{\partial C_{\eta}} \left( \bar{f} D_n \bar{g} - \bar{g} D_n \bar{f} \right) {\rm d} s(e) = \int_{C_{\eta}} (\bar{f} \Delta \bar{g} - \bar{g} \Delta \bar{f}) {\rm d} e \geq 0,
	\end{align*}
	where $\partial C_{\eta} = S_{0,\eta} \cup S_{0,1}$ is the boundary of $C_{\eta}$ and $D_n \bar{f}$ is the directional derivative of $\bar{f}$ in the direction of the outer normal unit vector.
	On the other hand, since $D_n \bar{g} = -(2-n) \eta^{1-n}$ on $S_{0,\eta}$ and $D_n \bar{g} = 2-n$ on $S_{0,1}$,
	\begin{align*}
	\int_{\partial C_{\eta}} \bar{f} D_n \bar{g} {\rm d} s(e) &= - (2-n) \eta^{1-n} \int_{S_{0,\eta}} \bar{f} {\rm d} s(e) + (2-n) \int_{S_{0,1}} \bar{f} {\rm d} s(e) \\
    &= - (2-n) \Omega_n L(f: X, \eta \rho) + (2-n) \Omega_n L(f: X, \rho).	    
    \end{align*}
	Also, since $\bar{g}=\eta^{2-n}-1$ on $S_{0,\eta}$ and $\bar{g}=0$ on $S_{0,1}$,
	\begin{align*}
	\int_{\partial C_{\eta}} \bar{g} D_n \bar{f} {\rm d} s(e) = (\eta^{2-n}-1) \int_{S_{0,\eta}} D_n \bar{f} {\rm d} s(e),
	\end{align*}
	which is $O(\eta)$ as $\eta \to 0$.
	Therefore, 
	\begin{align*}
	L(f: X, \rho) \leq L(f: X, \eta \rho) + O(\eta).
	\end{align*}
	By taking $\eta \to 0$, we obtain $L(f: X, \rho) \leq f(X)$.
	Since $X$ and $\rho$ are arbitrary, $f$ is matrix superharmonic.
\end{proof}

The limit of an increasing sequence of superharmonic functions is also superharmonic \cite[Theorem~3.4.8]{Helms}.
Matrix superharmonic functions have a similar property.

\begin{lemma}\label{lem_conv}
	Let $f_1 \leq f_2 \leq \dots$ be an increasing sequence of matrix superharmonic functions and assume that $f = \lim_{k \to \infty} f_k  \not\equiv \infty$.
	Then, $f$ is also matrix superharmonic.
\end{lemma}
\begin{proof}
	Since each $f_k$ is lower semicontinuous, their supremum $f$ is also lower semicontinuous.
    Also, from Lemma~3.2.10 of \cite{Helms},
	\begin{align*}
    \int_{S_{0,1}} f(X+e \rho^{\top}) {\rm d} s(e) = \lim_{k \to \infty} \int_{S_{0,1}} f_k(X+e \rho^{\top}) {\rm d} s(e).
	\end{align*}
	Therefore,
	\begin{align*}
	L(f: X, \rho) = \lim_{k \to \infty} L(f_k: X, \rho) \leq \lim_{k \to \infty} f_k(X) = f(X),
	\end{align*}
    where we used the matrix superharmonicity of each $f_k$.
\end{proof}

\section{Bayes estimation with matrix superharmonic prior}
In this section, we investigate Bayes shrinkage estimation under the matrix quadratic loss.

\subsection{Uniqueness of Bayes estimator}
Let
\begin{align}\label{marginal}
m_{\pi} (X) = \int p(X \mid M) \pi (M) {\rm d} M
\end{align}
be the marginal distribution of $X \sim {\rm N}_{n,p} (M,I_n,I_p)$ with prior $\pi(M)$.

Similarly to the Frobenius loss, the (generalized) Bayes estimator under the matrix quadratic loss is uniquely given by the posterior mean, even though the matrix quadratic loss is only partially ordered.

\begin{definition}
	For a function $f: \mathbb{R}^{n \times p} \to \mathbb{R}$, its \textrm{matrix gradient} $\widetilde{\nabla} f: \mathbb{R}^{n \times p} \to \mathbb{R}^{n \times p}$ is defined as 
	\begin{align*}
	(\widetilde{\nabla} f(X))_{ai} = \frac{\partial}{\partial X_{ai}} f(X).
	\end{align*}
\end{definition}

\begin{lemma}
	If $m_{\pi}(X)<\infty$ for every $X$, then the (generalized) Bayes estimator of $M$ with respect to a prior $\pi(M)$ under the matrix quadratic loss is uniquely given by the posterior mean:
	\begin{align*}
	\hat{M}^{\pi} (X) = \mathrm{E}_{\pi}[M \mid X] = \frac{\int M p(X \mid M) \pi(M) {\rm d} M}{\int p(X \mid M) \pi(M) {\rm d} M} = X + \widetilde{\nabla} \log m_{\pi}(X).
	\end{align*}
\end{lemma}
\begin{proof}
	For each estimator $\hat{M}=\hat{M}(X)$, its posterior risk is decomposed as
	\begin{align*}
	&{\rm E}_{\pi} [(\hat{M}(X)-M)^{\top} (\hat{M}(X)-M) \mid X] \\
	=& {\rm E}_{\pi} [(\mathrm{E}_{\pi}[M \mid X]-M)^{\top} (\mathrm{E}_{\pi}[M \mid X]-M)] + D^{\top} D,
	\end{align*}
	where $D = \hat{M}(X)-\mathrm{E}_{\pi}[M \mid X]$.
	Thus, the posterior risk is uniquely minimized under the Lowner order by taking $D=O$, which means that $\hat{M}(X)=\mathrm{E}_{\pi}[M \mid X]$.
\end{proof}

\subsection{Sufficient condition for minimaxity}
Now, we provide a sufficient condition for minimaxity of (generalized) Bayes estimators under the matrix quadratic loss.

\begin{theorem}\label{th_minimax}
    If $\sqrt{m_{\pi}(X)}$ is matrix superharmonic, then the generalized Bayes estimator $\hat{M}^{\pi}(X)=X+\widetilde{\nabla} \log m_{\pi}(X)$ with respect to $\pi(M)$ is minimax under the matrix quadratic loss.
\end{theorem}
\begin{proof}
	From Theorem~\ref{th_sure}, the matrix quadratic risk of the generalized Bayes estimator $\hat{M}^{\pi}$ is given by
	\begin{align*}
	R(M,\hat{M}^{\pi}) &= n I_p + {\rm E}_M [2 \widetilde{\Delta} \log m_{\pi}(X) + (\widetilde{\nabla} \log m_{\pi}(X))^{\top} (\widetilde{\nabla} \log m_{\pi}(X))]  \\
	&= n I_p  + 4 {\rm E}_M \left[ \frac{\widetilde{\Delta} \sqrt{m_{\pi}(X)}}{\sqrt{m_{\pi}(X)}} \right].
	\end{align*}
	Since $\sqrt{m_{\pi}(X)}$ is matrix superharmonic and $C^2$ by definition, $\widetilde{\Delta} \sqrt{m_{\pi}(X)} \preceq O$ for every $X$ from Theorem~\ref{th_c2}.
	Therefore, $R(M,\hat{M}^{\pi}) \preceq n I_p$ and thus $\hat{M}^{\pi}$ is minimax.
\end{proof}

We also provide another sufficient condition based on the matrix superharmonicity of prior itself.

\begin{lemma}\label{lem_m}
	If $\pi(M)$ is matrix superharmonic and $m_{\pi}(X) < \infty$ for every $X$, then ${m_{\pi}(X)}$ is also matrix superharmonic.
\end{lemma}
\begin{proof}
Let $\phi(Z)$ be the probability density of $Z \sim {\rm N}_{n,p}(O,I_n,I_p)$.
Then, 
\begin{align*}
	m_{\pi} (Z) &= \int \phi (Z-M) \pi (M) {\rm d} M
	= \int \phi (A) \pi (Z-A) {\rm d} A.
\end{align*}
Thus, for every $X \in \mathbb{R}^{n \times p}$ and $\rho \in \mathbb{R}^p$,
\begin{align*}
    L(m_{\pi}: X,\rho) =& \frac{1}{\Omega_n} \int_{S_{0,1}} \left\{ \int \phi (A) \pi (X+e \rho^{\top}-A) {\rm d} A \right\} {\rm d} s(e) \\
	=& \frac{1}{\Omega_n} \int \phi (A) \left\{ \int_{S_{0,1}} \pi (X+e \rho^{\top}-A) {\rm d} s(e) \right\} {\rm d} A \\
	=& \int \phi (A) L(\pi: X-A, \rho) {\rm d} A \\
	\leq & \int \phi (A) \pi(X-A) {\rm d} A \\
	= & m_{\pi} (X),
\end{align*}
where the second equation follows from Fubini's theorem and the inequality follows from the matrix superharmonicity of $\pi$.
Therefore, $m_{\pi}$ is matrix superharmonic.
\end{proof}

\begin{lemma}\label{lem_sqrt}
	If $f: \mathbb{R}^{n \times p} \to \mathbb{R}$ is matrix superharmonic and $\phi: \mathbb{R} \to \mathbb{R}$ is monotone increasing and concave in the range of $f$, then $\phi \circ {f}$ is also matrix superharmonic.
	In particular, if $f$ is matrix superharmonic and non-negative, then $\sqrt{f}$ is also matrix superharmonic.
\end{lemma}
\begin{proof}
For every $X \in \mathbb{R}^{n \times p}$ and $\rho \in \mathbb{R}^p$,
\begin{align*}
L(\phi \circ {f}: X, \rho) &= \frac{1}{\Omega_n} \int_{S_{0,1}} \phi(f(X+e \rho^{\top})) {\rm d} s(e) \\
&\leq \phi \left( \frac{1}{\Omega_n} \int_{S_{0,1}} {f}(X+e \rho^{\top}) {\rm d} s(e) \right) \\
&= \phi (L({f}: X, \rho)) \\
&\leq \phi(f(X)),
\end{align*}
where the first inequality follows from Jensen's inequality and the second inequality follows from the matrix superharmonicity of $f$ and the monotonicity of $\phi$.
Therefore, $\phi \circ {f}$ is matrix superharmonic.
\end{proof}

\begin{theorem}\label{th_minimax2}
    Let $\pi(M)$ be a superharmonic prior with $m_{\pi}(X) < \infty$ for every $X$.
    Then, the generalized Bayes estimator $\hat{M}^{\pi}(X)=X+\widetilde{\nabla} \log m_{\pi}(X)$ with respect to $\pi(M)$ is minimax under the matrix quadratic loss.
\end{theorem}
\begin{proof}
	From Lemma~\ref{lem_m} and Lemma~\ref{lem_sqrt}, $\sqrt{m_{\pi}(X)}$ is matrix superharmonic.
	Therefore, from Theorem~\ref{th_minimax}, $\hat{M}^{\pi}$ is minimax.
\end{proof}

When $p=1$, Theorem~\ref{th_minimax} and Theorem~\ref{th_minimax2} reduce to the classical results by Stein \cite{Stein74} on minimax Bayes estimators of a normal mean vector.

\begin{remark}
	Whereas matrix superharmonic priors provide minimax Bayes estimators of any linear combination $M c$ simultaneously, there may be cases where we are only interested in linear combination with nonnegative coefficients $c \in \mathbb{R}_+^p$.
	In such cases, only the copositivity \cite{Berman} of $-\widetilde{\Delta} m_{\pi}(X)$ suffices.
	Thus, it may be interesting to develop another version of matrix superharmonicity based on copositivity.
\end{remark}

\begin{remark}
Recently, superharmonic priors have been found to give minimax predictive densities under the Kullback--Leibler loss \cite{Komaki06,George06}.
It is an interesting future problem to investigate properties of matrix superharmonic priors in predictive density estimation under some analogue of matrix quadratic loss.
\end{remark}

\subsection{Matrix superharmonic priors}
Here, we provide a class of matrix superharmonic priors, which includes the previously proposed generalization \eqref{SVS} of Stein's prior.

Let
\begin{align}
\pi_{\alpha,\beta} (M)=\det (M^{\top} M + \beta I_p)^{-(\alpha+n+p-1)/2}, \label{mat_t}
\end{align} 
where $-n-p+1 \leq \alpha \leq -2p$ and $\beta \geq 0$.
When $\alpha = -2p$ and $\beta = 0$, $\pi_{\alpha,\beta} (M)$ coincides with the singular value shrinkage prior \eqref{SVS}.
Note that $\pi_{\alpha,\beta} (M)$ with $\alpha>0$ and $\beta>0$ is the so-called matrix t-distribution with $\alpha$ degrees of freedom \cite{Gupta}.
When $p=1$, $\pi_{\alpha,\beta} (M)$ reduces to the $n$-dimensional (improper) multivariate t-prior \cite{Faith}
\begin{align*}
\pi_{\alpha,\beta} (\mu)= (\| \mu \|^2 + \beta)^{-(\alpha+n)/2},
\end{align*} 
and it is superharmonic when $-n \leq \alpha \leq -2$ and $\beta \geq 0$.
This result is extended to general $p$ as follows.

\begin{theorem}\label{th_svs}
	If $-n-p+1 \leq \alpha \leq -2p$ and $\beta \geq 0$, then the prior $\pi_{\alpha,\beta}(M)$ in \eqref{mat_t} is matrix superharmonic.
\end{theorem}
\begin{proof}
In the following, the subscripts $a$, $b$, $\ldots$ run from $1$ to $n$ and the subscripts $i$, $j$, $\ldots$ run from $1$ to $p$.
We denote the $(i,j)$-th entry of $S^{-1}$ by $S^{ij}$ and the Kronecker delta by $\delta_{ij}$.

First, assume that $\beta > 0$.
Let $S = M^{\top} M + \beta I_p \succ O$.
Since $S_{ij} = \sum_a M_{a i} M_{a j} + \beta \delta_{i j}$,
\begin{align}
\frac{\partial S_{kl}}{\partial M_{a i}} = \delta_{ik} M_{al} + \delta_{il} M_{ak}. \label{sigma_m}
\end{align}
Also,
\begin{align*}
\frac{\partial}{\partial S_{ij}} \det S = S^{ij} \det S.
\end{align*}
Thus,
\begin{align}
\frac{\partial}{\partial M_{ai}} \det S & =
\sum_{k,l} \frac{\partial S_{kl}}{\partial M_{ai}} \frac{\partial}{\partial S_{kl}} \det S
= 2 \sum_k M_{ak} S^{ik} \det S. \label{det_1st_diff}
\end{align}
Therefore,
\begin{align}
&\frac{\partial^2}{\partial M_{ai} \partial M_{aj}} \det S \nonumber \\
=\ & 2 \left( S^{ij} - \sum_{k,l} M_{ak} M_{al} S^{ij} S^{kl} + \sum_{k,l} M_{ak} M_{al} S^{ik} S^{jl} \right) \det S, \label{det_2nd_diff}
\end{align}
where we used
\begin{align*}
\frac{\partial S^{ik}}{\partial M_{aj}} = -\sum_{l,m} S^{il} S^{km} \frac{\partial S_{lm}}{\partial M_{aj}} = -\sum_l M_{al} S^{ij} S^{kl} - \sum_l M_{al} S^{il} S^{jk},
\end{align*}
which is obtained by differentiating $\sum_k S_{jk} S^{ik} = \delta_{ij}$ and using \eqref{sigma_m}.

Now, 
\begin{align}
(\widetilde{\Delta} (\det S)^{-(\alpha+n+p-1)/2})_{ij} & = \sum_{a} \frac{\partial^2}{\partial M_{ai} \partial M_{aj}} (\det S)^{-(\alpha+n+p-1)/2} \nonumber \\
& = \frac{\alpha+n+p-1}{2} (\det S)^{-(\alpha+n+p-1)/2} \sum_{a} (A_{aij} + B_{aij}), \label{Delta_K}
\end{align}
where
\begin{align*}
A_{aij} &= \frac{\alpha+n+p+1}{2} (\det S)^{-2} \left( \frac{\partial}{\partial M_{ai}} \det S \right) \left( \frac{\partial}{\partial M_{aj}} \det S \right), \\
B_{aij} &= - (\det S)^{-1} \frac{\partial^2}{\partial M_{ai} \partial M_{aj}} \det S.
\end{align*}
By using $\eqref{det_1st_diff}$ and $\sum_a M_{ak} M_{al} = S_{kl} - \beta \delta_{kl}$,
\begin{align*}
\sum_a A_{aij} &= 2 (\alpha+n+p+1) \sum_a \left( \sum_k M_{ak} S^{ik} \right) \left( \sum_l M_{al} S^{jl} \right)\\
&= 2 (\alpha+n+p+1) \left( S^{ij} - \beta \sum_{k} S^{ik} S^{jk} \right).
\end{align*}
On the other hand, from $\eqref{det_2nd_diff}$,
\begin{align*}
\sum_a B_{aij} = -2 (n-p+1) S^{ij} -2 \beta S^{ij} \sum_k S^{kk} + 2 \beta \sum_{k} S^{ik} S^{jk}.
\end{align*}
Hence, 
\begin{align*}
\sum_a (A_{aij} + B_{aij}) = 2(\alpha+2p) S^{ij} -2(\alpha+n+p) \beta \sum_{k} S^{ik} S^{jk} -2 \beta S^{ij} \sum_k S^{kk}. %
\end{align*}
Substituting this expression into \eqref{Delta_K} gives
\begin{align}
\widetilde{\Delta} (\det S)^{-(\alpha+n+p-1)/2} &= \frac{\alpha+n+p-1}{2} (\det S)^{-(\alpha+n+p-1)/2} \nonumber \\ 
&\times (2(\alpha+2p) S^{-1} -2(\alpha+n+p) \beta (S^{-1})^2 -2 \beta{\rm tr} (S^{-1}) S^{-1}), \label{laplacian}
\end{align}
which is negative semidefinite from $-n-p+1 \leq \alpha \leq -2p$, $\beta > 0$ and $S^{-1} \succ O$.
Therefore, by Theorem~\ref{th_c2}, $\pi_{\alpha,\beta}(M)=(\det S)^{-(\alpha+n+p-1)/2}$ is matrix superharmonic.

Next, assume that $\beta = 0$.
Let 
\begin{align*}
    \pi^{(k)} (M) = \det \left( M^{\top} M + k^{-1} I_p \right)^{-(\alpha+n+p-1)/2}, \quad k=1,2,\dots.
\end{align*}
Then, each $\pi^{(k)}$ is matrix superharmonic from the above discussion.
Also, $\pi^{(1)} \leq \pi^{(2)} \leq \cdots$ and $\lim_{k \rightarrow \infty} \pi^{(k)}(M) = \pi_{\alpha,\beta} (M)$ for every $M$.
Therefore, from Lemma~\ref{lem_conv}, $\pi_{\alpha,\beta}(M)$ is also matrix superharmonic.
\end{proof}

\begin{proposition}\label{prop_marginal}
	If $-n-p+1 \leq \alpha \leq -2p$ and $\beta \geq 0$, then the marginal density $m_{\pi}(X)$ in \eqref{marginal} with $\pi=\pi_{\alpha,\beta}$ in \eqref{mat_t} is finite for every $X$.
\end{proposition}
\begin{proof}
We use the fact that $m_{\pi} (X)$ in \eqref{marginal} is interpreted as the expected value of $\pi(M)$ with respect to $M \sim {\rm N}_{n,p} (X, I_n, I_p)$.

When $\beta > 0$, since $\pi_{\alpha,\beta}(M) \leq \pi_{\alpha,\beta}(O)=\beta^{-p(\alpha+n+p-1)/2}$,
\begin{align*}
	m_{\pi}(X) \leq \beta^{-p(\alpha+n+p-1)/2}
\end{align*}
for every $X$.

When $\beta = 0$, 
\begin{align*}
	m_{\pi} (X) = \mathrm{E} \left[ (\det S)^{-(\alpha+n+p-1)/2} \right],
\end{align*}
where $S = M^{\top} M$ has a noncentral Wishart distribution $S \sim W_p (n, I_p, X^{\top} X)$ from Theorem 3.5.1 in \cite{Gupta}.
Therefore, by using Theorem 3.5.6 of \cite{Gupta}, 
\begin{align*}
	m_{\pi} (X) = C {\rm etr} \left( -\frac{1}{2} X^{\top} X \right)
	{}_1 F_1 \left( -\frac{\alpha+p-1}{2}; \frac{n}{2}; \frac{1}{2} X^{\top} X \right),
\end{align*}
where $C={2^{-p(\alpha+n+p-1)/2} \Gamma_p (-(\alpha+p-1)/2)}/{\Gamma_p ( {n/2} )}$, $\Gamma_p$ is the multivariate Gamma function and ${}_1 F_1$ is the hypergeometric function of a matrix argument \cite{Gupta}.
Thus, $m_{\pi} (X)$ is finite for every $X$.
\end{proof}

From Theorem~\ref{th_minimax2}, Theorem~\ref{th_svs} and Proposition~\ref{prop_marginal}, we obtain the following.

\begin{theorem}\label{th_svs_minimax}
	If $-n-p+1 \leq \alpha \leq -2p$ and $\beta \geq 0$, then the generalized Bayes estimator with respect to  the prior $\pi_{\alpha,\beta}(M)$ in \eqref{mat_t} is minimax under the matrix quadratic loss.
\end{theorem}

By taking $\alpha = -2p$ and $\beta = 0$ in Theorem~\ref{th_svs} and Theorem~\ref{th_svs_minimax}, we obtain the following result on the singular value shrinkage prior \eqref{SVS}.

\begin{corollary}\label{cor_svs}
	When $n-p-1>0$, the singular value shrinkage prior $\pi_{\mathrm{SVS}}(M)$ in \eqref{SVS} is matrix superharmonic.
	Also, the generalized Bayes estimator with respect to  $\pi_{\mathrm{SVS}}(M)$ is minimax under the matrix quadratic loss.
\end{corollary}

In particular, the matrix superharmonicity of $\pi_{{\rm SVS}}$ is strongly concentrated on the space of low rank matrices, which has measure zero, in the same way as the Laplacian of Stein's prior $\pi(\mu)=\| \mu \|^{2-n}$ becomes a Dirac delta function at the origin.

\begin{corollary}\label{cor_svs2}
    If $M$ has full-rank, then $\widetilde{\Delta} \pi_{\mathrm{SVS}}(M) = O$.
\end{corollary}
\begin{proof}
	Substituting $\alpha=-2p$ and $\beta=0$ into \eqref{laplacian} gives
\begin{align*}
	\widetilde{\Delta} (\det S)^{-(n-p-1)/2} = - (n-p-1) (\det S)^{-(n-p-1)/2}  \{ (n-p) (S^{-1})^2 + {\rm tr} (S^{-1}) S^{-1} \} \beta = O,
\end{align*}
where we used $\det S > 0$ since $M$ has full-rank.
\end{proof}

\begin{remark}
The Efron--Morris estimator \eqref{EM_estimator} can be viewed as a pseudo-Bayes estimator $\hat{M}=X+\widetilde{\nabla} \log m(X)$ \cite{shr_book} with the pseudo-marginal $m(X)=\pi_{\mathrm{SVS}}(X)=\det (X^{\top} X)^{-(n-p-1)/2}$.
Combining such a pseudo-Bayes interpretation with Theorem~\ref{th_minimax} and Theorem~\ref{th_svs} with $\beta=0$, it follows that the estimator $\hat{M}=X(I_p-c(X^{\top}X)^{-1})$ is minimax for $0 \leq c \leq 2(n-p-1)$.
\end{remark}

\subsection{Stein-type priors}
We further investigate matrix superharmonicity of another types of shrinkage priors.

Since Stein's prior $\pi(\mu)=\| \mu \|^{2-n}$ for $\mu \in \mathbb{R}^n$ is superharmonic \cite{Stein74}, the prior 
\begin{align}\label{Stein_prior}
    \pi_{\mathrm{S}}(M)=\| M \|_{\mathrm{F}}^{2-np}=\| {\rm vec}(M) \|^{2-np}
\end{align}
is also superharmonic.
More generally, the shrinkage prior $\pi(M)=\| M \|^{-c}$ with $c \geq 0$ is superharmonic if and only if $0 \leq c \leq np-2$ \cite{shr_book}.
However, the range of $c$ for matrix superharmoncity is narrower.
In particular, Stein's prior $\pi_{\mathrm{S}}(M)$ in \eqref{Stein_prior} is not matrix superharmonic.

\begin{proposition}\label{prop_stein}
The prior $\pi(M)=\| M \|_{\mathrm{F}}^{-c}$ with $c \geq 0$ is matrix superharmonic if and only if $0 \leq c \leq n-2$.
\end{proposition}
\begin{proof}
First, assume that $0 \leq c \leq n-2$.
Let $\pi_{\beta}(M)=(\| M \|_{\mathrm{F}}^2+\beta)^{-c/2}$ with $\beta>0$.
Then, $\pi_{\beta}(M)$ is $C^2$ and its matrix Laplacian is
\begin{align*}
    \widetilde{\Delta} \pi_{\beta}(M) = -c (\| M \|_{\mathrm{F}}^2+\beta)^{-c/2-2} (n (\| M \|_{\mathrm{F}}^2+\beta)I_p - (c+2) M^{\top} M) \preceq O,
\end{align*}
where we used $n (\| M \|_{\mathrm{F}}^2+\beta)I_p - (c+2) M^{\top} M \succeq O$ from $n \geq c+2$ and $\| M \|_{\mathrm{F}}^2 I_p \succeq M^{\top} M$.
Therefore, from Theorem~\ref{th_c2}, $\pi_{\beta}(M)$ is matrix superharmonic for every $\beta>0$.
Then, $\pi^{(k)}(M)=(\| M \|_{\mathrm{F}}^2+k^{-1})^{-c/2}$ is an increasing sequence of matrix superharmonic functions with $\lim_{k \to \infty} \pi^{(k)}(M)=\pi(M)$.
Thus, from Lemma~\ref{lem_conv}, $\pi(M)$ is also matrix superharmonic.

Next, assume that $c > n-2$.
Consider $M$ and $\rho$ defined by
\begin{align*}
    M_{ai} = \begin{cases} 1 & (i=1) \\ 0 & (2 \leq i \leq p) \end{cases}, \quad \rho_i = \begin{cases} 1 & (i=1) \\ 0 & (2 \leq i \leq p) \end{cases}.
\end{align*}
Then,
\begin{align*}
L(\pi: M, \rho)= \frac{1}{\Omega_n} \int_{S_{0,1}} \pi(M+e \rho^{\top}) {\rm d} s(e) = \frac{1}{\Omega_n} \int_{S_{0,1}} g(e) {\rm d} s(e),
\end{align*}
where the function $g: \mathbb{R}^n \to \mathbb{R}$ is given by $g(e)=\| \mathbf{1}+e \|^{-c}$ with the all-one vector $\mathbf{1}=(1,\dots,1)^{\top} \in \mathbb{R}^n$.
Since $c > n-2$, $\Delta g(e) = c(c-n+2) \| \mathbf{1}+e \|^{-c-2} > 0$.
Thus, by Green's theorem, 
\begin{align*}
\frac{1}{\Omega_n} \int_{S_{0,1}} g(e) {\rm d} s(e) > g(0) = \pi(M).
\end{align*}
Hence, we have $L(\pi: M, \rho) > \pi(M)$.
Therefore, $\pi(M)$ is not matrix superharmonic.
\end{proof}

\begin{corollary}\label{cor_stein}
	When $p \geq 2$, Stein's prior $\pi_{\mathrm{S}}(M)=\| M \|_{\mathrm{F}}^{2-np}$ is not matrix superharmonic.
\end{corollary}

\cite{abu} showed that the column-wise shrinkage estimator of James--Stein type
\begin{align*}
\hat{M} = X D, \quad D = {\rm diag}(d_1,\dots,d_p), \quad d_i = 1-\frac{c}{\sum_a X_{ai}^2}
\end{align*}
is minimax under the matrix quadratic loss when $0 \leq c \leq 2(n-2)/p$.
Note that this estimator can be viewed as a pseudo-Bayes estimator $\hat{M}=X+\widetilde{\nabla} \log m(X)$ \cite{shr_book} with the pseudo-marginal $m(X)=\prod_i (\sum_a X_{ai}^2)^{-c/2}$.
Their result is understood from the viewpoint of matrix superharmonicity as follows.

\begin{proposition}\label{prop_abu}
	The prior $\pi(M)=\prod_i (\sum_a M_{ai}^2)^{-c/2}$ with $c \geq 0$ is matrix superharmonic if and only if $0 \leq c \leq (n-2)/p$.
\end{proposition}
\begin{proof}
First, assume that $0 \leq c \leq (n-2)/p$.
Let $\pi_{\beta}(M)=\prod_i (\sum_a M_{ai}^2 + \beta)^{-c/2}$ with $\beta>0$.
Then, $\pi_{\beta}(M)$ is $C^2$ and its matrix Laplacian is
\begin{align*}
    \widetilde{\Delta} \pi_{\beta}(M) = c \pi_{\beta}(M) (c A - (n-2) B - n \beta C),
\end{align*}
where $A,B,C$ are $p \times p$ positive semidefinite matrices with entries
\begin{align*}
    A_{ij}=\left( \sum_a M_{ai}^2 + \beta \right)^{-1} \left( \sum_a M_{aj}^2 + \beta \right)^{-1} \left( \sum_a M_{ai} M_{aj} \right), 
\end{align*}
\begin{align*}
    B_{ij}=\delta_{ij} A_{ij}, \quad C_{ij}=\delta_{ij} \left( \sum_a M_{ai}^2 + \beta \right)^{-2}. 
\end{align*}
Let $S=B^{-1/2}AB^{-1/2} \succeq O$. 
Then, all diagonal entries of $S$ are one and thus $S \preceq ({\rm tr} S) I_p = p I_p$.
Thus, $A = B^{1/2} S B^{1/2} \preceq p B$.
Therefore, from $c \leq (n-2)/p$,
\begin{align*}
    \widetilde{\Delta} \pi_{\beta}(M) \preceq c (cp-n+2) \pi_{\beta}(M) B \preceq O.
\end{align*}
Hence, from Theorem~\ref{th_c2}, $\pi_{\beta}(M)$ is matrix superharmonic for every $\beta>0$.
Then, $\pi^{(k)}(M)=\prod_i (\sum_a M_{ai}^2 + k^{-1})^{-c/2}$ is an increasing sequence of matrix superharmonic functions with $\lim_{k \to \infty} \pi^{(k)}(M)=\pi(M)$.
Thus, from Lemma~\ref{lem_conv}, $\pi(M)$ is also matrix superharmonic.

Next, assume that $c > (n-2)/p$.
Suppose that all entries of $M$ and $\rho$ are one.
Then,
\begin{align*}
L(\pi: M, \rho)= \frac{1}{\Omega_n} \int_{S_{0,1}} \pi(M+e \rho^{\top}) {\rm d} s(e) = \frac{1}{\Omega_n} \int_{S_{0,1}} g(e) {\rm d} s(e),
\end{align*}
where the function $g: \mathbb{R}^n \to \mathbb{R}$ is given by $g(e)=\| \mathbf{1}+e \|^{-cp}$ with the all-one vector $\mathbf{1}=(1,\dots,1)^{\top} \in \mathbb{R}^n$.
Since $c > (n-2)/p$, $\Delta g(e) = cp(cp-n+2) \| \mathbf{1}+e \|^{-cp-2} > 0$.
Thus, by Green's theorem, 
\begin{align*}
\frac{1}{\Omega_n} \int_{S_{0,1}} g(e) {\rm d} s(e) > g(0) = \pi(M).
\end{align*}
Hence, we have $L(\pi: M, \rho) > \pi(M)$.
Therefore, $\pi(M)$ is not matrix superharmonic.
\end{proof}

\section{Numerical results}
In this section, we present several numerical results on the matrix quadratic risk of shrinkage estimators. %
We denote the $i$-th singular value of $M$ by $\sigma_i$.
Note that singular values are in descending order: $\sigma_1 \geq \sigma_2 \geq \cdots \geq \sigma_p$.

In the following, we focus on the eigenvalues $\lambda_1 \geq \cdots \geq \lambda_p$ of the matrix quadratic risk $R(M,\hat{M})$ of several estimators.
Since $R(M,\hat{M})=n I_p$ for the maximum likelihood estimator $\hat{M}=X$, an estimator is minimax if and only if $\lambda_1 \leq n$ for every $M$.

First, we compare the generalized Bayes estimators with respect to the singular value shrinkage prior $\pi_{\mathrm{SVS}}(M)$ in \eqref{SVS} and Stein's prior $\pi_{\mathrm{S}}(M)$ in \eqref{Stein_prior} for $n=5$ and $p=3$. 
We employed the numerical method of \cite{Matsuda} to compute the generalized Bayes estimators. 

Figure~\ref{fig1} plots the three eigenvalues $\lambda_1 \geq \lambda_2 \geq \lambda_3$ of the matrix quadratic risk with respect to $\sigma_2$ when $\sigma_1=10$ and $\sigma_3=0$.
For $\pi_{{\rm SVS}}$, all eigenvalues do not exceed $n=5$, which indicates the minimaxity.
More specifically, both $\lambda_1$ and $\lambda_3$ are almost constant with values $\lambda_1 \approx 5$ and $\lambda_3 \approx 4$ respectively, whereas $\lambda_2$ increases from 4 to 5 with $\sigma_2$.
These behaviors are understood from the fact that $\pi_{{\rm SVS}}$ shrinks each singular value separately \cite{Matsuda}.
For $\pi_{{\rm S}}$, $\lambda_1$ is larger than $n=5$ when $\sigma_2 \leq 8$ and thus the estimator is not minimax.
This is compatible with Proposition~\ref{prop_stein}.
However, the sum $\lambda_1+\lambda_2+\lambda_3$ of eigenvalues, which is equal to the Frobenius risk ${\rm E}_M [\| \hat{M}-M \|_{\mathrm{F}}^2 ] = {\rm tr} \ R(M,\hat{M})$, does not exceed $np=15$, because $\pi_{{\rm S}}$ is superharmonic in usual sense.
This is similar to the fact that the James--Stein estimator is not minimax componentwise, even though it is minimax under the quadratic loss for the whole vector \cite{Lehmann}.

\begin{figure}[htbp]
\centering
\begin{tikzpicture}
\tikzstyle{every node}=[]
\begin{axis}[width=7cm,
xmax=10,xmin=0,
ymax=6, ymin=0,
xlabel={$\sigma_2$},ylabel={eigenvalue},
ylabel near ticks,
	]
\addplot[thin, dashed, color=black,
filter discard warning=false, unbounded coords=discard
] table {
         0    5
   10.0000    5
};
\addplot[very thick, color=black,filter discard warning=false, unbounded coords=discard] table {
         0    4.0051
1.0000    4.0166
2.0000    4.0296
3.0000    4.0079
4.0000    4.0103
5.0000    4.0076
6.0000    4.0276
7.0000    4.0121
8.0000    4.0221
9.0000    4.0065
10.0000    4.0128

};
\addplot[very thick, color=black,filter discard warning=false, unbounded coords=discard] table {
         0    4.0186
1.0000    4.1646
2.0000    4.4293
3.0000    4.6606
4.0000    4.8145
5.0000    4.9023
6.0000    4.9356
7.0000    4.9486
8.0000    4.9648
9.0000    4.9694
10.0000    4.9675

};
\addplot[very thick, color=black,filter discard warning=false, unbounded coords=discard] table {
         0    4.9959
1.0000    4.9831
2.0000    4.9923
3.0000    4.9951
4.0000    4.9881
5.0000    4.9810
6.0000    4.9725
7.0000    4.9676
8.0000    4.9876
9.0000    4.9892
10.0000    4.9933

};
\end{axis}
\end{tikzpicture} 
\begin{tikzpicture}
\tikzstyle{every node}=[]
\begin{axis}[width=7cm,
xmax=10,xmin=0,
ymax=6, ymin = 0,
xlabel={$\sigma_2$},ylabel={eigenvalue},
ylabel near ticks,
	]
\addplot[thin, dashed, color=black,
filter discard warning=false, unbounded coords=discard
] table {
         0    5
   10.0000    5
};
\addplot[very thick, color=black,filter discard warning=false, unbounded coords=discard] table {
         0    3.9068
1.0000    3.9242
2.0000    3.9601
3.0000    3.9849
4.0000    4.0416
5.0000    4.0990
6.0000    4.1813
7.0000    4.2299
8.0000    4.3008
9.0000    4.3470
10.0000    4.4101

};
\addplot[very thick, color=black,filter discard warning=false, unbounded coords=discard] table {
         0    3.9164
1.0000    3.9654
2.0000    4.0353
3.0000    4.1292
4.0000    4.2611
5.0000    4.4079
6.0000    4.5314
7.0000    4.6405
8.0000    4.7437
9.0000    4.8325
10.0000    4.8800

};
\addplot[very thick, color=black,filter discard warning=false, unbounded coords=discard] table {
         0    5.6653
1.0000    5.6211
2.0000    5.5731
3.0000    5.4827
4.0000    5.3658
5.0000    5.2617
6.0000    5.1572
7.0000    5.0569
8.0000    5.0073
9.0000    4.9296
10.0000    4.9035

};
\end{axis}
\end{tikzpicture} 

	\caption{Eigenvalues of matrix quadratic risk of generalized Bayes estimators ($n=5$, $p=3$, $\sigma_1=10$, $\sigma_3=0$). left: $\pi_{\mathrm{SVS}}$. right: $\pi_{\mathrm{S}}$. The dashed line shows $n=5$.}
	\label{fig1}
\end{figure}
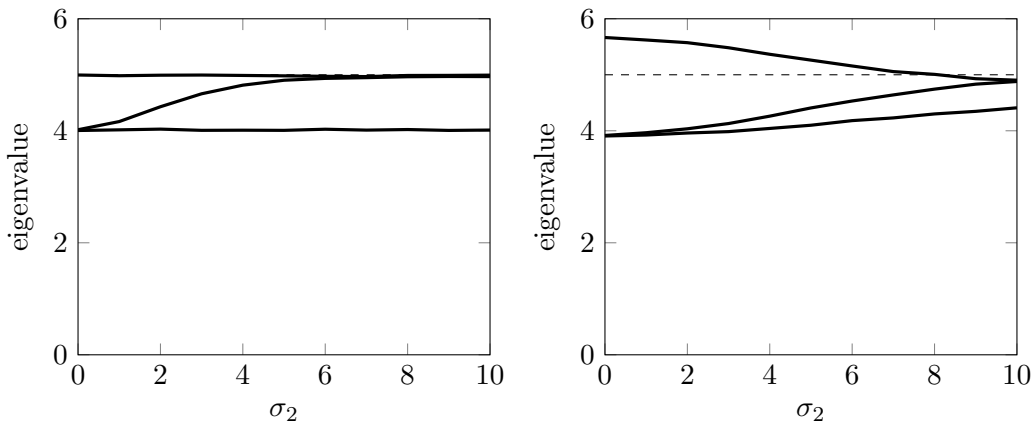

Figure~\ref{fig2} plots the three eigenvalues $\lambda_1 \geq \lambda_2 \geq \lambda_3$ of the matrix quadratic risk with respect to $\sigma_1$ when $\sigma_2=\sigma_3=0$.
Thus, the rank of $M$ is one.
For $\pi_{{\rm SVS}}$, both $\lambda_2$ and $\lambda_3$ are almost constant around 4, whereas $\lambda_1$ increases from 4 to 5 with $\sigma_1$.
It indicates that $\pi_{{\rm SVS}}$ works particularly well when $M$ has low rank. 
For $\pi_{{\rm S}}$, all eigenvalues are fairly small when $\sigma_1=0$, namely $M=O$.
However, $\lambda_1$ increases rapidly with $\sigma_1$ and becomes larger than $n=5$ when $\sigma_1 \geq 4$.

\begin{figure}[htbp]
\centering
\begin{tikzpicture}
\tikzstyle{every node}=[]
\begin{axis}[width=7cm,
xmax=10,xmin=0,
ymax=6, ymin=0,
xlabel={$\sigma_1$},ylabel={eigenvalue},
ylabel near ticks,
	]
\addplot[thin, dashed, color=black,
filter discard warning=false, unbounded coords=discard
] table {
         0    5
   10.0000    5
};
\addplot[very thick, color=black,filter discard warning=false, unbounded coords=discard] table {
         0    4.0069
1.0000    4.0192
2.0000    4.0200
3.0000    4.0088
4.0000    4.0170
5.0000    4.0165
6.0000    4.0234
7.0000    4.0166
8.0000    4.0207
9.0000    4.0164
10.0000    4.0114

};
\addplot[very thick, color=black,filter discard warning=false, unbounded coords=discard] table {
         0    4.0243
1.0000    4.0524
2.0000    4.0585
3.0000    4.0311
4.0000    4.0292
5.0000    4.0345
6.0000    4.0417
7.0000    4.0210
8.0000    4.0295
9.0000    4.0285
10.0000    4.0251

};
\addplot[very thick, color=black,filter discard warning=false, unbounded coords=discard] table {
         0    4.0562
1.0000    4.1362
2.0000    4.3664
3.0000    4.5956
4.0000    4.7446
5.0000    4.8320
6.0000    4.8809
7.0000    4.9092
8.0000    4.9501
9.0000    4.9552
10.0000    4.9728

};
\end{axis}
\end{tikzpicture} 
\begin{tikzpicture}
\tikzstyle{every node}=[]
\begin{axis}[width=7cm,
xmax=10,xmin=0,
ymax=6, ymin = 0,
xlabel={$\sigma_1$},ylabel={eigenvalue},
ylabel near ticks,
	]
\addplot[thin, dashed, color=black,
filter discard warning=false, unbounded coords=discard
] table {
         0    5
   10.0000    5
};
\addplot[thin, dashed, color=black,
filter discard warning=false, unbounded coords=discard
] table {
         0    5
   10.0000    5
};
\addplot[very thick, color=black,filter discard warning=false, unbounded coords=discard] table {
         0    0.6632
1.0000    0.7324
2.0000    0.9357
3.0000    1.2834
4.0000    1.7612
5.0000    2.2692
6.0000    2.7491
7.0000    3.1443
8.0000    3.4658
9.0000    3.7154
10.0000    3.9129
};
\addplot[very thick, color=black,filter discard warning=false, unbounded coords=discard] table {
         0    0.6672
1.0000    0.7390
2.0000    0.9459
3.0000    1.2903
4.0000    1.7664
5.0000    2.2802
6.0000    2.7622
7.0000    3.1487
8.0000    3.4737
9.0000    3.7271
10.0000    3.9257

};
\addplot[very thick, color=black,filter discard warning=false, unbounded coords=discard] table {
         0    0.6738
1.0000    1.1676
2.0000    2.4595
3.0000    3.9841
4.0000    5.1489
5.0000    5.7569
6.0000    5.9290
7.0000    5.8901
8.0000    5.8360
9.0000    5.7148
10.0000    5.6323
};
\end{axis}
\end{tikzpicture} 
	\caption{Eigenvalues of matrix quadratic risk of generalized Bayes estimators ($n=5$, $p=3$, $\sigma_2=\sigma_3=0$). left: $\pi_{\mathrm{SVS}}$. right: $\pi_{\mathrm{S}}$.  The dashed line shows $n=5$. Note that the second and third eigenvalues almost overlap in both plots.}
	\label{fig2}
\end{figure}
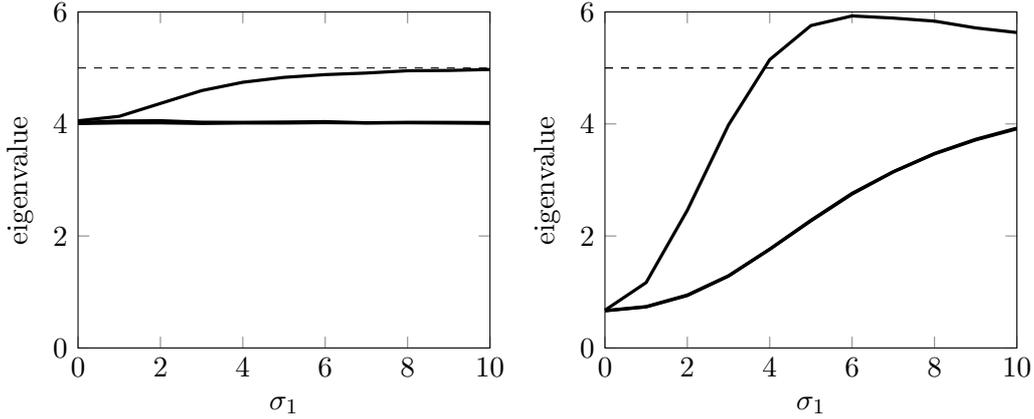

Next, we compare the Efron--Morris estimator $\hat{M}_{\mathrm{EM}}=X(I-(n-p-1) (X^{\top} X)^{-1})$ and the James--Stein estimator $\hat{M}_{\mathrm{JS}}=(1-(np-2)/\| X \|_{\mathrm{F}}^2)X$ in higher dimension.
Note that these estimators have almost the same risk with the generalized Bayes estimators with respect to $\pi_{{\rm SVS}}$ and $\pi_{{\rm S}}$, respectively.

Figure~\ref{fig3} plots the 20 eigenvalues $\lambda_1 \geq \cdots \geq \lambda_{20}$ of the matrix quadratic risk with respect to $\sigma_1$ when $n=100$, $p=20$, $\sigma_i=(6-i)/5 \cdot \sigma_1$ $(i=2,\dots,5)$ and $\sigma_6=\dots=\sigma_{20}=0$.
Thus, the rank of $M$ is five. 
The results are qualitatively the same with Figure~\ref{fig2}, with the advantage of the singular value shrinkage more pronounced.
For $\hat{M}_{\mathrm{EM}}$, all eigenvalues are smaller than $n=100$.
In particular, the bottom 15 eigenvalues are almost constant around 20: $\lambda_{6} \approx \cdots \approx \lambda_{20} \approx 20$.
On the other hand, for $\hat{M}_{\mathrm{JS}}$, the largest eigenvalue $\lambda_1$ grows rapidly with $\sigma_1$ and exceeds $n=100$ when $\sigma_1 \geq 10$.
Other eigenvalues also increase with $\sigma_1$, including $\lambda_{6}, \dots, \lambda_{20}$.
These results show that the singular value shrinkage works well for low rank matrices.

	\input{fig3.tex}
	
For the Efron--Morris estimator, the above simulation results suggest that the $i$-th eigenvalue $\lambda_i$ of the matrix quadratic risk depends only on the $i$-th singular value $\sigma_i$ of $M$ approximately: $\lambda_i \approx g_{n,p}(\sigma_i)$ for some function $g_{n,p}$.
Finally, we investigate this functional relation $g_{n,p}$ numerically.
Figure~\ref{fig4} plots $\lambda_1$ with respect to $\sigma_1$ when $\sigma_2=\dots=\sigma_p=0$ for several values of $n$ and $p$.
It indicates that $g_{n,p}(0) \approx p$, which is compatible with $R(O,\hat{M}_{\mathrm{EM}}) = (p+1) I_p$ from Corollary~\ref{cor_EM}.
In addition, Figure~\ref{fig4} implies $g_{n,p}(\sigma) \to n$ as $\sigma \to \infty$.
This is understood from the fact that the Efron--Morris estimator $\hat{M}_{\mathrm{EM}}$ becomes essentially the same with the maximum likelihood estimator $\hat{M}=X$, which has the constant risk $R(M,\hat{M}) = n I_p$, in a direction of a very large singular value \cite{Stein74,Matsuda}. 
These properties of $g_{n,p}$ provide a quantification of the advantage of the Efron--Morris estimator over the maximum likelihood estimator when $M$ has low rank.
Namely, if $M$ has rank $r<p$, then $\sigma_{r+1}=\cdots=\sigma_p=0$ and thus $\lambda_{r+1},\dots,\lambda_p$ should be close to $g_{n,p}(0) \approx p$.
Therefore, the reduction in the Frobenius risk is evaluated as 
\begin{align*}
    np-{\rm tr} \ R(M,\hat{M}_{\mathrm{EM}}) \geq np - rn - (p-r)p = np \left( 1-\frac{r}{p} \right) \left(1-\frac{p}{n} \right).
\end{align*}
Thus, the Efron--Morris estimator attains large risk reduction especially when either $p/r$ or $n/p$ is large.

\begin{figure}[htbp]
\centering
\begin{tikzpicture}
\tikzstyle{every node}=[]
\begin{axis}[width=7cm,
xmax=30,xmin=0,
ymax=110, ymin=0,
xlabel={$\sigma_1$},ylabel={eigenvalue},
ylabel near ticks,
	]
\addplot[thin, dashed, color=black,
filter discard warning=false, unbounded coords=discard
] table {
         0    100
   30.0000    100
};
\addplot[very thick, color=black,filter discard warning=false, unbounded coords=discard] table {

ans =

         0   11.1974
    1.0000   11.9084
    2.0000   14.4810
    3.0000   18.3512
    4.0000   23.3183
    5.0000   28.8730
    6.0000   34.6941
    7.0000   40.5008
    8.0000   45.9553
    9.0000   50.9981
   10.0000   55.6554
   11.0000   59.9499
   12.0000   63.8238
   13.0000   67.1157
   14.0000   70.1878
   15.0000   72.8127
   16.0000   75.4635
   17.0000   77.3605
   18.0000   79.0797
   19.0000   80.9386
   20.0000   82.5246
   21.0000   83.9171
   22.0000   84.8184
   23.0000   85.8277
   24.0000   87.0132
   25.0000   88.0199
   26.0000   88.8291
   27.0000   89.1806
   28.0000   90.0996
   29.0000   90.6695
   30.0000   91.1175

};
\addplot[very thick, color=black,filter discard warning=false, unbounded coords=discard] table {

ans =

         0   21.3701
    1.0000   21.7333
    2.0000   24.0049
    3.0000   27.6794
    4.0000   31.9907
    5.0000   36.8373
    6.0000   41.9883
    7.0000   47.0656
    8.0000   51.9502
    9.0000   56.5843
   10.0000   60.6619
   11.0000   64.4665
   12.0000   67.9900
   13.0000   70.7586
   14.0000   73.5911
   15.0000   75.9175
   16.0000   77.9588
   17.0000   80.0377
   18.0000   81.5969
   19.0000   83.1173
   20.0000   84.2924
   21.0000   85.4577
   22.0000   86.9514
   23.0000   87.8169
   24.0000   88.3410
   25.0000   89.0445
   26.0000   89.8787
   27.0000   90.8432
   28.0000   91.2027
   29.0000   91.7510
   30.0000   92.2820

};
\addplot[very thick, color=black,filter discard warning=false, unbounded coords=discard] table {

ans =

         0   31.5386
    1.0000   31.9944
    2.0000   33.6626
    3.0000   36.7732
    4.0000   40.6499
    5.0000   44.7673
    6.0000   49.3539
    7.0000   53.6926
    8.0000   58.0472
    9.0000   62.1785
   10.0000   65.7416
   11.0000   68.9465
   12.0000   71.9578
   13.0000   74.6543
   14.0000   76.7943
   15.0000   79.0254
   16.0000   80.7577
   17.0000   82.4683
   18.0000   83.8904
   19.0000   85.4900
   20.0000   86.2326
   21.0000   87.5020
   22.0000   88.3617
   23.0000   89.4502
   24.0000   89.7344
   25.0000   90.5004
   26.0000   91.1211
   27.0000   91.8843
   28.0000   92.3118
   29.0000   92.7392
   30.0000   93.1386

};
\addplot[very thick, color=black,filter discard warning=false, unbounded coords=discard] table {

ans =

         0   41.7986
    1.0000   41.7997
    2.0000   43.5120
    3.0000   45.8651
    4.0000   49.1916
    5.0000   52.8760
    6.0000   56.6265
    7.0000   60.7385
    8.0000   64.2768
    9.0000   67.4982
   10.0000   70.7581
   11.0000   73.3758
   12.0000   75.9317
   13.0000   78.2296
   14.0000   80.2160
   15.0000   81.9166
   16.0000   83.6526
   17.0000   85.1347
   18.0000   86.2276
   19.0000   87.3822
   20.0000   88.4110
   21.0000   89.2712
   22.0000   89.9992
   23.0000   90.7752
   24.0000   91.3813
   25.0000   92.0411
   26.0000   92.5118
   27.0000   92.9188
   28.0000   93.5175
   29.0000   93.8212
   30.0000   94.1382

};
\addplot[very thick, color=black,filter discard warning=false, unbounded coords=discard] table {

ans =

         0   52.0397
    1.0000   51.9833
    2.0000   52.9483
    3.0000   55.0897
    4.0000   57.7221
    5.0000   60.7878
    6.0000   63.9889
    7.0000   67.1676
    8.0000   70.0104
    9.0000   73.1407
   10.0000   75.5794
   11.0000   78.0206
   12.0000   79.9496
   13.0000   81.9382
   14.0000   83.6451
   15.0000   85.1560
   16.0000   86.4389
   17.0000   87.7262
   18.0000   88.5181
   19.0000   89.4079
   20.0000   90.1603
   21.0000   91.0883
   22.0000   91.5010
   23.0000   92.2521
   24.0000   92.8736
   25.0000   93.3448
   26.0000   93.9792
   27.0000   94.1911
   28.0000   94.5696
   29.0000   95.0500
   30.0000   95.1520

};
\end{axis}
\end{tikzpicture} 
\begin{tikzpicture}
\tikzstyle{every node}=[]
\begin{axis}[width=7cm,
xmax=100,xmin=0,
ymax=1100, ymin = 0,
xlabel={$\sigma_1$},ylabel={eigenvalue},
ylabel near ticks,
	]
\addplot[thin, dashed, color=black,
filter discard warning=false, unbounded coords=discard
] table {
         0    1000
   100.0000    1000
};
\addplot[very thick, color=black,
filter discard warning=false, unbounded coords=discard
] table {
         0  102.9150
    5.0000  122.9424
   10.0000  182.4874
   15.0000  266.0985
   20.0000  358.0096
   25.0000  446.6412
   30.0000  527.1198
   35.0000  595.7522
   40.0000  654.4065
   45.0000  702.9785
   50.0000  743.0993
   55.0000  777.0996
   60.0000  805.1100
   65.0000  827.5551
   70.0000  848.0654
   75.0000  865.0142
   80.0000  878.6218
   85.0000  890.5596
   90.0000  901.6521
   95.0000  910.4038
  100.0000  918.2384
};
\addplot[very thick, color=black,
filter discard warning=false, unbounded coords=discard
] table {
         0  205.0250
    5.0000  220.5226
   10.0000  273.6145
   15.0000  347.9302
   20.0000  429.4305
   25.0000  508.6499
   30.0000  580.0406
   35.0000  641.1476
   40.0000  692.4526
   45.0000  736.8180
   50.0000  772.2710
   55.0000  801.7196
   60.0000  827.1215
   65.0000  846.6474
   70.0000  864.7445
   75.0000  879.6644
   80.0000  892.5706
   85.0000  903.5432
   90.0000  912.4401
   95.0000  920.8700
  100.0000  927.8206
};
\addplot[very thick, color=black,
filter discard warning=false, unbounded coords=discard
] table {
         0  306.9390
    5.0000  318.4629
   10.0000  365.0158
   15.0000  429.1102
   20.0000  500.7877
   25.0000  570.1737
   30.0000  632.0691
   35.0000  686.3663
   40.0000  731.2966
   45.0000  768.9628
   50.0000  800.3949
   55.0000  826.6500
   60.0000  848.3529
   65.0000  866.2411
   70.0000  881.2346
   75.0000  894.3587
   80.0000  905.4103
   85.0000  915.6277
   90.0000  923.7018
   95.0000  930.9497
  100.0000  936.0608
};
\addplot[very thick, color=black,
filter discard warning=false, unbounded coords=discard
] table {
         0  409.1166
    5.0000  417.0775
   10.0000  456.0781
   15.0000  511.4504
   20.0000  572.2593
   25.0000  631.3804
   30.0000  684.5915
   35.0000  730.4222
   40.0000  770.3247
   45.0000  802.3645
   50.0000  829.4483
   55.0000  850.8155
   60.0000  870.2514
   65.0000  884.6795
   70.0000  899.0987
   75.0000  909.6312
   80.0000  918.2161
   85.0000  927.2496
   90.0000  934.3329
   95.0000  940.0774
  100.0000  945.7491
};
\addplot[very thick, color=black,
filter discard warning=false, unbounded coords=discard
] table {
         0  511.2952
    5.0000  515.1762
   10.0000  546.4629
   15.0000  593.4353
   20.0000  644.1145
   25.0000  693.1004
   30.0000  737.6283
   35.0000  776.0212
   40.0000  808.7994
   45.0000  834.7852
   50.0000  857.6820
   55.0000  876.3285
   60.0000  891.8634
   65.0000  905.2051
   70.0000  915.4261
   75.0000  925.6541
   80.0000  932.7582
   85.0000  939.2766
   90.0000  945.1541
   95.0000  949.7294
  100.0000  954.7877
};
\end{axis}
\end{tikzpicture} 
	\caption{Largest eigenvalue of matrix quadratic risk of the Efron--Morris estimator ($\sigma_2=\dots=\sigma_p=0$). left: $n=100$, $p=10,20,\dots,50$. right: $n=1000$, $p=100,200,\dots,500$. The dashed line shows $n$.}
	\label{fig4}
\end{figure}
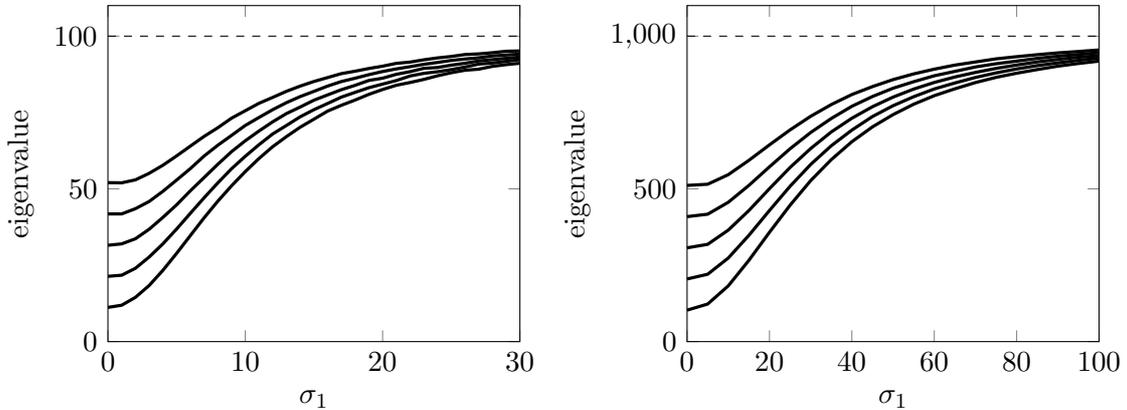

\section*{Acknowledgements}
We thank Yuzo Maruyama for helpful comments.
This work was supported by JSPS KAKENHI Grant Number 19K20220.
This work was partially funded by grant \# 418098 from the Simons Foundation to William Strawderman.


\begin{thebibliography}{99}	
	\bibitem[{Abu-Shanab et al.(2012)}]{abu}
	\textsc{Abu-Shanab, R.}, \textsc{Kent, J. T.} \& \textsc{Strawderman, W. E.} (2012).
	{Shrinkage estimation with a matrix loss function}.
	 \textit{Electronic Journal of Statistics} \textbf{6}, 2347--2355.
	 
	\bibitem[{Berman and Shaked-Mondered(2003)}]{Berman}
	\textsc{Berman, A.} \& \textsc{Shaked-Monderer, N.} (2003).
	 \textit{Completely Positive Matrices}.
	 World Scientific.

	\bibitem[{Bilodeau and Kariya(1989)}]{Bilodeau}
	\textsc{Bilodeau, M.} \& \textsc{Kariya, T.} (1989).
	{Minimax estimators in the normal MANOVA model}.
	 \textit{Journal of Multivariate Analysis} \textbf{28}, 260--270.
	
	\bibitem[{Efron and Morris(1972)}]{Efron72}
	\textsc{Efron, B.} \& \textsc{Morris, C.} (1972).
	{Empirical Bayes on vector observations: an extension of Stein's method}.
	 \textit{Biometrika} \textbf{59}, 335--347.
	
	\bibitem[{Faith(1993)}]{Faith}
	\textsc{Faith, M.} (1993).
	{Minimax Bayes estimators of a multivariate normal mean}.
	\textit{Journal of Multivariate Analysis} \textbf{8}, 372--379.

	\bibitem[{Fourdrinier et al.(2018)}]{shr_book}
	\textsc{Fourdrinier, D.}, \textsc{Strawderman, W. E.} \& \textsc{Wells, M.} (2018).
	 \textit{Shrinkage Estimation}.
	 New York: Springer-Verlag.
	
	\bibitem[{George et al.(2006)}]{George06}
    \textsc{George, E. I.}, \textsc{Liang, F.} \& \textsc{Xu, X.} (2006).
    {Improved minimax predictive densities under Kullback--Leibler loss}.
    \textit{Annals of Statistics} \textbf{34}, 78--91.

    \bibitem[{Gupta and Nagar(2000)}]{Gupta}
	\textsc{Gupta, A. K.} \& \textsc{Nagar, D. K.} (2000).
	 \textit{Matrix Variate Distributions}.
	 New York: Chapman \& Hall.

    \bibitem[{Helms(2014)}]{Helms}
	\textsc{Helms, L. L.} (2014).
	 \textit{Potential Theory}.
	 New York: Springer-Verlag.
	
	\bibitem[{Honda(1991)}]{Honda}
	\textsc{Honda, T.} (1991).
	{Minimax estimators in the MANOVA model for arbitrary quadratic loss and unknown covariance matrix}.
	 \textit{Journal of Multivariate Analysis} \textbf{36}, 113--120.

\bibitem[{Komaki(2006)}]{Komaki06}
\textsc{Komaki, F.} (2006).
{Shrinkage priors for Bayesian prediction}.
 \textit{Annals of Statistics} \textbf{34}, 808--819.

	\bibitem[{Lehmann and Casella(2006)}]{Lehmann}
	\textsc{Lehmann, E. L.} \& \textsc{Casella, G.} (2006).
	\textit{Theory of Point Estimation}.
	New York: Springer-Verlag.

    \bibitem[{Magnus and Neudecker(2019)}]{Magnus}
	\textsc{Magnus, J. R.} \& \textsc{Neudecker, H.} (2019).
	\textit{Matrix Differential Calculus with Applications in Statistics and Econometrics}.
	New York: Wiley.

	\bibitem[{Matsuda and Komaki(2015)}]{Matsuda}
	\textsc{Matsuda, T.} \& \textsc{Komaki, F.} (2015).
	{Singular value shrinkage priors for Bayesian prediction}.
	 \textit{Biometrika} \textbf{102}, 843--854.

	\bibitem[{Matsuda and Komaki(2019)}]{Matsuda2}
\textsc{Matsuda, T.} \& \textsc{Komaki, F.} (2019).
{Empirical Bayes matrix completion}.
\textit{Computational Statistics \& Data Analysis} \textbf{137}, 195--210.

	\bibitem[{Murota(2003)}]{Murota}
	\textsc{Murota, K.} (2003).
	 \textit{Discrete Convex Analysis}.
	 Society for Industrial and Applied Mathematics.

	\bibitem[{Stein(1974)}]{Stein74}
	\textsc{Stein, C.} (1974).
	{Estimation of the mean of a multivariate normal distribution}.
	 \textit{Proc. Prague Symp. Asymptotic Statistics} \textbf{2}, 345--381.

	\bibitem[{Xie(1993)}]{Xie}
	\textsc{Xie, M.} (1993).
	{All admissible linear estimates of the mean matrix}.
	\textit{Journal of Multivariate Analysis} \textbf{44}, 220--226.
\end{thebibliography}
\end{document}